\documentclass[12pt]{article}
\usepackage{times}
\usepackage{hyperref}

\usepackage{graphicx}%

\usepackage{latexsym}
\usepackage{amsmath}
\usepackage[left=1in,top=1in,right=1in]{geometry}
\usepackage{amssymb,amsmath}
\usepackage{amsthm}
\usepackage{amsmath}
\usepackage{amsfonts}
\usepackage{titlesec}
\titlelabel{\thetitle.\quad}

\providecommand{\U}[1]{\protect\rule{.1in}{.1in}}
\newtheorem{theorem}{Theorem}

\newtheorem{definition}[theorem]{Definition}
\newtheorem{example}[theorem]{Example}

\title{\textbf{ {\Large  Interpolation of surfaces with asymptotic curves in Euclidean 3-space} }}
\author{\textbf{Mustafa Alt\i n} \\Technical Sciences Vocational School,Bing\"ol University, Turkey\\ maltin@bingol.edu.tr\\ \\ \textbf{\.Inan \"Unal}\\ Department of Computer Engineering\\Munzur University, Tunceli, Turkey\\ inanunal@munzur.edu.tr\\\\
\textbf{Fatemah Mofarreh} \\Mathematical Science Department \\Faculty of Science Princess Nourah bint Abdulrahman University\\ Riyadh 11546 Saudi Arabia.\\ fyalmofarrah@pnu.edu.sa\\ }

\date{ }
\begin{document}
\maketitle

\begin{abstract}
	In this paper, we investigate the interpolation of surfaces which are obtained from an isoasymptotic curve in 3D-Euclidean space. We prove that there exist a unique $ C^0 $-Hermite surface interpolation related to an isoasymptotic curve  under some special conditions on the marching scale functions. Finally, we present some examples and plot their graphs. 
\end{abstract}

\section{Introduction}
Differential geometry is a branch of mathematics which uses advanced calculus tools in geometry. In recent years, it has become an applicable area of mathematics in science and technology. Since the manifold theory has used in general relativity in the 1900s by Einstein, differential geometry of curves, surfaces and general manifolds have been improving more. Notably, from medicine to social science and from artificial intelligence to economy, it is very clear how the differential geometry is applied. In this manner, one can consider that the applied mathematics has been changed from numerical and computational methods to differential geometrical tools. For example, to understand the meaning of multiple features data we use calculus on manifolds in machine learning \cite{bronstein2017geometric}. Also, differential geometry permit us to work on non-euclidean spaces as most real life problems are define in such spaces. So, it is a fundamental tool for understanding events in the universe. \par 
A relevant kind of curves is geodesics which play the role of straight lines in Euclidean space on a manifold. Gauss proved that the differential geometry of a surface is different from the geometry of ambient space. The well known example of supporting these ideas is that a geodesic on a unit sphere embedded in a Euclidean space is not a geodesic in a Euclidean space. In this way, the differential geometry of a surface has many significant properties that we can use in applied sciences. The minimal distance between two points on a surface is called a geodesic and this is considered as an important idea in many applications \cite{bremond1994estimation,haw1985application}. 
\par
 A surface could be constructed using a geodesic.  In \cite{6} a general surface is obtained from a polynomial geodesic. Also, considering a 3-dimensional polynomial curve which is a pregeodesic, the authors constructed ruled cubic patched in \cite{5}. In addition, in \cite{4}  authors investigated a developable surface that contains a given Bezier geodesic. Wang at al. defined a parametric surface which is called  surface family, using a geodesic curve \cite{wang2004parametric}.  They used the Frenet frame of the curve and presented necessary conditions in which the curve is an isogeodesic on a parametric surface considering Frenet aparatus of the curve. The methods in this paper is a reverse of some engineering problems. Later, Kasap et al. \cite{kasap2008generalization} generalized their methods and presented examples. Li et al. studied the approximation minimal surface with geodesics by using the Dirichlet function and they minimized the area of a surface family by using Dirichlet approach. This method can be used for obtaining the  minimal cost of the material while building surfaces.  The family surfaces have been studied for example as  in \cite{kasap2006surfaces,kasap2008generalization,csaffak2009family,atalay2015surfaces,yuzbacsi2016construction,yuzbasi2016family}. \par 
The other special curve which is as important as geodesics is the asymptotic curves.  An asymptotic curve is a curve always tangent to an asymptotic direction of the surface and it has zero normal curvature. On the other hand, one can be constructed a surface using an asymptotic curve.  Saad et al. \cite{saad2019minimal} approximated the minimal parametric surface with an asymptotic curve by minimizing the Dirichlet function. In \cite{liu2013designing} they examined rational developable surface pencils through an arbitrary parametric curve as its common asymptotic curve. Moreover,  G\"uler at al. constructed  a surface interpolating a given curve as the asymptotic curve of it \cite{guler2019asym}. Similar to   geodesics, the asymptotic curves also have many applications in related sciences.  In \cite{schling2018designing} the authors presented a method to design strained grid structures along asymptotic curves to benefit from a high degree of simplification in fabrication and construction.  Also, asymptotic curves has many applications in astronomy as in \cite{contopoulos1990asymptotic,efthymiopoulos1999cantori}. 
 \par 
Lee et al. \cite{leeinterpolation2020} introduced a new method to construct a parametric surface in terms of curves. They defined a surface interpolation associated with a spatial curve passing through some $ m$-points in Euclidean $ 3 $-space. Motivating all studies mentioned from here, we consider a surface interpolation using asymptotic  curves in Euclidean space. We organized this paper as follow: We give some fundamental facts which are used throughout the paper, in Section 2. Section 3 is devoted to the surface interpolations with isoparametric curve and examples with their graphs. 
\section{Preliminaries}
In this section, we give a short review on curves and surfaces in an 3D-Euclidean space. For details  we refer the reader to any classical differential geometry books, for example \cite{struik1961lectures}.\par 
Let  $\gamma(\omega)$ be a curve which is arc-length $\omega$ in 3D Euclidean space.  Take the Frenet frame of  $\gamma(\omega)$ by $\left\{  \mathcal{V}_1(\omega),\mathcal{V}_2(\omega),\mathcal{V}_3(\omega)\right\}  $. Then we have the following well known relations between $\kappa(\omega)$ and $\tau(\omega)$ are the curvature and the torsion of the curve
$\gamma(\omega)$, respectively.

\[
\left[
\begin{array}
[c]{c}%
{\mathcal{V}_1}^{^{\prime}}(\omega)\\
{\mathcal{V}_2}{^{\prime}}(\omega)\\
{\mathcal{V}_3}^{^{\prime}}(\omega)
\end{array}
\right]  =\left[
\begin{array}
[c]{ccc}%
0 & \kappa(\omega) & 0\\
-\kappa(\omega) & 0 & \tau(\omega)\\
0 & -\tau(\omega) & 0
\end{array}
\right]  .\left[
\begin{array}
[c]{c}%
\mathcal{V}_1(\omega)\\
\mathcal{V}_2(\omega)\\
\mathcal{V}_3(\omega)
\end{array}
\right]
\].

Previous equations are called the Frenet apparatus of a curve and they are important  to understand the geometry of the curve. Also, we can classify curves via the Frenet–Serret frames. In \cite{wang2004parametric} Wang et al. defined pencil surface which could be obtained using the Frenet–Serret frames  of the curve. This surface is called by a surface family or a pencil surface and it is defined as follows. 

\begin{definition}
	Let  $\gamma(\omega)$ be a curve which is arc-length $\omega$ in $ \mathbb{E}^3 $ and $\left\{  \mathcal{V}_1(\omega),\mathcal{V}_2(\omega),\mathcal{V}_3(\omega)\right\}$ be the Frenet frame of $ \gamma $. Then the map 
	\begin{equation}
	\Psi(\omega,\eta)=\gamma(\omega)+u(\omega,\eta)\cdot\mathcal{V}_1(\omega)+v(\omega,\eta)\cdot\mathcal{V}_2(\omega)+z(\omega,\eta)\cdot\mathcal{V}_3(\omega),\label{ana}%
	\end{equation}
	is defined a surface in $ \mathbb{E}^3 $, where  $ \Omega\geq \omega\geq0, \Lambda\geq \eta\geq0 $ for a real-valued constants $\Omega  , \Lambda $, and \ $u(\omega,\eta),v(\omega,\eta)$ and $z(\omega,\eta)$ are $C^{1}-$functions. The surface $ 	\Psi(\omega,\eta) $ is called as surface family or pencil surface \cite{wang2004parametric}. 
\end{definition}  
By the following definition we classify some special curves on a parametric surface $\Psi(\omega,\eta)$.
\begin{definition}
	Lets take a curve $\gamma(s)$ on a parametric surface $X(s,\eta)$. Then we have following characterizations  \cite{struik1961lectures}:
	\begin{enumerate}
			\item [$\bullet$] $\gamma(s)$ is said to be an isoparametric curve on $X(s,\eta)$ if there exists a parameter $\eta_{0}%
		\in[0,\Lambda]$ such that $X(s,\eta_{0})=\gamma(s).$
		\item [$\bullet$] $\gamma(s)$ is an asymptotic curve on a parametric surface $X(s,\eta)$
		if $\frac{\partial \mathcal{N}(s,\eta_{0})}{\partial s}.\mathcal{V}_1(s)=0,$ where $\mathcal{N}(s,\eta)$
		is the normal vector of surface $X(s,\eta)$, $\mathcal{V}_1(s)$ 	is the tangent vector of curve $\gamma(s)$.
		\item [$\bullet$] $\gamma(s)$ 	is called isoasymptotic of the surface $X(s,\eta)$ if it is both a asymptotic curve and  an
		isoparametric curve  on the surface $X(s,\eta)$.
	\end{enumerate}

\end{definition}
With following theorem we have the necessary and sufficient  conditions for $ \gamma $  to be isoasymptotic curve on the surface $\Psi(\omega,\eta)$  (\ref{ana}).

\begin{theorem}
\cite{bayram2012parametric} Let $\Psi(\omega,\eta)$ be a parametric surface with the marching-scale functions;  
\begin{align*}
u(\omega,\eta)  & =k(\omega)U(\eta),\\
v(\omega,\eta)  & =m(\omega)V(\eta),\\
z(\omega,\eta)  & =n(\omega)Z(\eta), 
\end{align*}
where $k(\omega), U(\eta), m(\omega), V(\eta), n(\omega)$ and $Z(\eta)$ are $C^{1}-$functions.
Then $\gamma(\omega)$ is an 
isoasymptotic curve on a parametric surface $\Psi(\omega,\eta)$ (\ref{ana}) if and only if we have 
\begin{equation}
\left\{
\begin{array}
[c]{c}%
U(\eta_{0})=V(\eta_{0})=Z(\eta_{0})=0,\\
n(\omega)=0\text{ or }\frac{\partial Z(\eta_{0})}{\partial \eta}=0.
\end{array}
\right.  \label{sart}
\end{equation}

\end{theorem}
If we take  $k(\omega)=m(\omega)=n(\omega)=1$ ,particularly  and consider $U(\eta)$, $V(\eta)$ and $Z(\eta)$
as polynomials of the forms in (\ref{sart}), then we have 
\begin{align}
U(\eta)  & =\sum\limits_{t=1}^{n}a_{t}(\eta-\eta_{0})^{t},\label{fonksiyon}\\
V(\eta)  & =\sum\limits_{t=1}^{n}b_{t}(\eta-\eta_{0})^{t},\nonumber\\
Z(\eta)  & =\sum\limits_{t=1}^{n}c_{t}(\eta-\eta_{0})^{t},\text{ \ }c_{1}=0,\nonumber
\end{align}
respectively, where $a_{t},b_{t},c_{t}$ are constant. Then the polynomials
$U(\eta)$, $V(\eta)$ and $Z(\eta)$ in (\ref{fonksiyon}) satisfy the isoasymptotic
condition (\ref{sart}). Thus, we can determine marching-scale functions of a surface family $\Psi(\omega,\eta)$  by the polynomial expressions. 

\section{Surface Interpolations with Isoasymptotic Curve}

In this section, we construct a surface with isogeodesic curve passing
through finite control points lying on $\mathbb{E}^{3}.$

Now, we give a definition for surface interpolations with isoasymptotic curve
passing through some control points on  $\mathbb{E}^{3}.$

\begin{definition}
Let $A_{1},A_{2},...,A_{m}$ be different points on $\mathbb{E}^{3}$ and
$\Psi(\omega,\eta):B\subset\mathbb{R}^{2}\rightarrow\mathbb{E}^{3}$ be a parametric pencil
surface given by (\ref{ana}). For some different points $(\omega_{t},\eta_{t})\in B$
$(t=1,...,m),$ we can construct the surface $\Psi(\omega,\eta)$ such that $\Psi
(\omega_{t},\eta_{t})=A_{t}$. It is called a surface interpolation associated with the
given isoasymptotic curve $\gamma(\omega)$ passing through $m$-control points
$A_{t}~(t=1,...,m)$, simply, $C^{0}-$Hermite surface interpolation with an
isoasymptotic curve.  In particular, $\left\{  A_{1},A_{2},...,A_{m}\right\}
$ is called $C^{0}$-Hermite data.
\end{definition}

Polynomials $U(\eta)$, $V(\eta)$ and $Z(\eta)$ with degree $n$ in (\ref{fonksiyon})
have $n,n$ and $n-1$ degrees of freedom in terms of coeffcients $a_{t},b_{t}$
and $c_{t}$ respectively. In this case, there are two extra degrees of
freedom. To determine a unique parametric surface with isoasymptotic curve, we
may assume $a_{n}=b_{n}=0$.

Now, we consider an isoasymptotic curve on surface parametrization%
\begin{equation}
\Psi(\omega,\eta)=\gamma(\omega)+U(\eta)\cdot\mathcal{V}_1(\omega)+V(\eta)\cdot\mathcal{V}_2(\omega)+Z(\eta)\cdot\mathcal{V}_3(\omega),\label{yeniY}%
\end{equation}%
\[
\Omega\geq \omega\geq0,\text{ \ \ \ \ \ \ }\Lambda\geq \eta\geq0,\text{\ \ }%
\]
with the marching-scale functions are given in (\ref{fonksiyon}) for $a_{n}=b_{n}=0.$ 

\begin{theorem}
Let $A_{1},A_{2},...,A_{m}$ be different points on a parametric surface
$\Psi(\omega,\eta)$ given in (\ref{yeniY}). For $\Psi(\omega_{t},\eta_{t})=A_{t},t=1,...,m$
there exists a unique $C^{0}$-Hermite surface interpolation with an
isoasymptotic curve such that the marching-scale functions are given by%
\begin{align*}
U(\eta)   =\sum\limits_{t=1}^{n-1}a_{t}(\eta-\eta_{0})^{t},\ \ 
V(\eta)   =\sum\limits_{t=1}^{n-1}b_{t}(\eta-\eta_{0})^{t}\ \ \text{and }
Z(\eta)  =\sum\limits_{t=2}^{n}c_{t}(\eta-\eta_{0})^{t},
\end{align*}
and 
\begin{align*}
det\left(
\begin{array}
[c]{ccccc}  
d_{12}&\eta_{1}-\eta_{0} & (\eta_{1}-\eta_{0})^{2} & \cdots & (\eta_{1}-\eta_{0})^{n-1}\\
d_{22}&\eta_{2}-\eta_{0} & (\eta_{2}-\eta_{0})^{2} & \cdots & (\eta_{2}-\eta_{0})^{n-1}\\
\vdots\vdots & \vdots & \ddots & \vdots\\
d_{n2}&\eta_{n-1}-\eta_{0} & (\eta_{n-1}-\eta_{0})^{2} & \cdots & (\eta_{n-1}-\eta_{0})^{n-1}%
\end{array}
\right)  
\end{align*}
where $a_{t},b_{t}$ and $c_{t}$ are constant and $ d_{t2}=V(\eta_{t2}), \ t=1,2, \ldots , n$.
\end{theorem}

\begin{proof}
	Let us define $(n-1)$-points of the surface $ \Psi(\omega_{t},\eta_{t}) $ by 
\[
\Psi(\omega_{t},\eta_{t})=A_{t}\text{ for }\omega\geq \eta_{n-1}...\geq \eta_{1}\geq \eta_{0}%
\geq0.\text{ }%
\]
So, we have  $\Psi(\omega_{t},\eta_{t})=A_{t}=\gamma(\omega_{t})+\mathcal{V}_1(\omega_{t}).U(\eta_{t})+\mathcal{V}_2(\omega_{t}%
).V(\eta_{t})+\mathcal{V}_3(\omega_{t}).Z(\eta_{t})$. By taking inner product with  $\mathcal{V}_1(\omega)$, $\mathcal{V}_2(\omega)$, and
$\mathcal{V}_3(\omega)$, respectively we obtain the coefficients as following 
\begin{align*}
U(\eta_{t})  & =\left\langle A_{t}-\gamma(\omega_{t}),\mathcal{V}_1(\omega_{t})\right\rangle ,\\
V(\eta_{t})  & =\left\langle A_{t}-\gamma(\omega_{t}),\mathcal{V}_2(\omega_{t})\right\rangle ,\\
Z(\eta_{t})  & =\left\langle A_{t}-\gamma(\omega_{t}),\mathcal{V}_3(\omega_{t})\right\rangle .
\end{align*}
Using 
\[
U(\eta_{t})=d_{t1},\text{ }V(\eta_{t})=d_{t2},\text{ }Z(\eta_{t})=d_{t3}%
\]
where $d_{t1}$, $d_{t2}$ and $d_{t3}$ are constant, from (\ref{fonksiyon}) for
$a_{n}=b_{n}=0,$ we have following matrices; 
\[
\left(
\begin{array}
[c]{cccc}%
\eta_{1}-\eta_{0} & (\eta_{1}-\eta_{0})^{2} & \cdots & (\eta_{1}-\eta_{0})^{n-1}\\
\eta_{2}-\eta_{0} & (\eta_{2}-\eta_{0})^{2} & \cdots & (\eta_{2}-\eta_{0})^{n-1}\\
\vdots & \vdots & \ddots & \vdots\\
\eta_{n-1}-\eta_{0} & (\eta_{n-1}-\eta_{0})^{2} & \cdots & (\eta_{n-1}-\eta_{0})^{n-1}%
\end{array}
\right)  \left(
\begin{array}
[c]{c}%
a_{1}\\
a_{2}\\
\vdots\\
a_{n-1}%
\end{array}
\right)  =\left(
\begin{array}
[c]{c}%
d_{11}\\
d_{21}\\
\vdots\\
d_{(n-1)1}%
\end{array}
\right)  ,
\]%
\[
\left(
\begin{array}
[c]{cccc}%
\eta_{1}-\eta_{0} & (\eta_{1}-\eta_{0})^{2} & \cdots & (\eta_{1}-\eta_{0})^{n-1}\\
\eta_{2}-\eta_{0} & (\eta_{2}-\eta_{0})^{2} & \cdots & (\eta_{2}-\eta_{0})^{n-1}\\
\vdots & \vdots & \ddots & \vdots\\
\eta_{n-1}-\eta_{0} & (\eta_{n-1}-\eta_{0})^{2} & \cdots & (\eta_{n-1}-\eta_{0})^{n-1}%
\end{array}
\right)  \left(
\begin{array}
[c]{c}%
b_{1}\\
b_{2}\\
\vdots\\
b_{n-1}%
\end{array}
\right)  =\left(
\begin{array}
[c]{c}%
d_{12}\\
d_{22}\\
\vdots\\
d_{(n-1)2}%
\end{array}
\right)  ,
\]%
\[
\left(
\begin{array}
[c]{cccc}%
(\eta_{1}-\eta_{0})^{2} & (\eta_{1}-\eta_{0})^{3} & \cdots & (\eta_{1}-\eta_{0})^{n}\\
(\eta_{2}-\eta_{0})^{2} & (\eta_{2}-\eta_{0})^{3} & \cdots & (\eta_{2}-\eta_{0})^{n}\\
\vdots & \vdots & \ddots & \vdots\\
(\eta_{n-1}-\eta_{0})^{2} & (\eta_{n-1}-\eta_{0})^{3} & \cdots & (\eta_{n-1}-\eta_{0})^{n}%
\end{array}
\right)  \left(
\begin{array}
[c]{c}%
c_{2}\\
c_{3}\\
\vdots\\
c_{n}%
\end{array}
\right)  =\left(
\begin{array}
[c]{c}%
d_{13}\\
d_{23}\\
\vdots\\
d_{(n-1)3}%
\end{array}
\right)  ,
\]
for $ 1\leq \eta \leq n-1$.  
Lets take 
\begin{align*}
M_{1}  & =\left(
\begin{array}
[c]{cccc}%
\eta_{1}-\eta_{0} & (\eta_{1}-\eta_{0})^{2} & \cdots & (\eta_{1}-\eta_{0})^{n-1}\\
\eta_{2}-\eta_{0} & (\eta_{2}-\eta_{0})^{2} & \cdots & (\eta_{2}-\eta_{0})^{n-1}\\
\vdots & \vdots & \ddots & \vdots\\
\eta_{n-1}-\eta_{0} & (\eta_{n-1}-\eta_{0})^{2} & \cdots & (\eta_{n-1}-\eta_{0})^{n-1}%
\end{array}
\right)  
\end{align*}
and
\begin{align*}
M_{2}  & =\left(
\begin{array}
[c]{cccc}%
(\eta_{1}-\eta_{0})^{2} & (\eta_{1}-\eta_{0})^{3} & \cdots & (\eta_{1}-\eta_{0})^{n}\\
(\eta_{2}-\eta_{0})^{2} & (\eta_{2}-\eta_{0})^{3} & \cdots & (\eta_{2}-\eta_{0})^{n}\\
\vdots & \vdots & \ddots & \vdots\\
(\eta_{n-1}-\eta_{0})^{2} & (\eta_{n-1}-\eta_{0})^{3} & \cdots & (\eta_{n-1}-\eta_{0})^{n}%
\end{array}
\right).
\end{align*}
Then the determinants of $ M_1 $ and $ M_2 $ is obtained as  
\begin{align*}
\det(M_{1})  & =(-1)^{\frac{(n-1)(n-2)}{2}}\prod\limits_{t=1}^{n-1}%
(\eta_{t}-\eta_{0})\prod\limits_{1\leq t<j\leq n-1}(\eta_{t}-\eta_{j}),
\end{align*}
and 
\begin{align*}
\det(M_{2})   =(-1)^{\frac{(n-1)(n-2)}{2}}\prod\limits_{t=1}^{n-1}%
(\eta_{t}-\eta_{0})^{2}\prod\limits_{1\leq t<j\leq n-1}(\eta_{t}-\eta_{j}). 
\end{align*}
Since $\eta_{t}$ and $\eta_{j}$ are non-zero and different from each others, for $1\leq
t<j\leq n-1$, we get $\det(M_{1})\neq0$ and $\det(M_{1})\neq0$, that is,
$a_{1},a_{2}$, ...,$a_{n-1},$ $b_{1},b_{2}$, ...,$b_{n-1}$ and $c_{2},c_{3}$,
...,$c_{n}$ have unique solutions. This gives us that there exists a uniquely $C^{0}%
$-Hermite surface interpolation with an isoasymptotic curve.
\end{proof}
\begin{example}
	Consider a curve parametrized by%
	\begin{equation}
	\gamma(\omega)=(\frac{\sin \omega}{2},\frac{\cos \omega}{2},\frac{\sqrt{3}\omega}{2}),\text{
		\ \ \ \ \ }0\leq \omega\leq2\pi. \label{xam1egri}%
	\end{equation}
	The curve (\ref{xam1egri}) is shown in Figure 1. By a direct computation,
	we have%
	\begin{align*}
	\mathcal{V}_1(\omega)  &  =(\frac{\cos \omega}{2},-\frac{\sin \omega}{2},\frac{\sqrt{3}}{2}),\\
	\mathcal{V}_2(\omega)  &  =(-\sin \omega,-\cos \omega,0),\\
	\mathcal{V}_3(\omega)  &  =(\frac{\sqrt{3}\cos \omega}{2},-\frac{\sqrt{3}\sin \omega}{2},\frac{-1}{2})
	\end{align*}
	For $A_{1}=(1,7,\frac{\sqrt{3}\pi}{4})$, the point $A_{1}$ lies on the surface
	pencil with an isoasymptotic curve given by (\ref{yeniY}). If we take%
	\[
	U(\eta)=a_{1}\eta,\text{ }V(\eta)=b_{1}\eta,\text{ }W(\eta)=c_{2}\eta^{2}%
	\]
	then there is only one surface with an isoasymptotic curve passing the point
	$A_{1}.$ We take $\omega_{1}=\frac{\pi}{2},$ $\eta_{1}=2,$ i.e., $\Psi_1(\frac{\pi}%
	{2},2)=A_{1}(1,7,\frac{\sqrt{3}\pi}{4})$ We obtain the equations:%
	\begin{align*}
	\frac{1}{2}-2b_{1}  &  =1,\\
	-a_{1}-2\sqrt{3}c_{2}  &  =7,\\
	\sqrt{3}a_{1}-c_{2}  &  =0,
	\end{align*}
	which imply%
	\[
	a_{1}=-1,\text{ }b_{1}=-\frac{1}{2},\text{ }c_{2}=-\sqrt{3}.
	\]
	Thus, we can construct the surface with an isoasymptotic curve passing the one
	point $A_{1}=(1,7,\frac{\sqrt{3}\pi}{4})$ given by%
	\begin{align}
	\Psi_1(\omega,\eta)  &  =\frac{1}{2}(\sin(\omega)-3\eta^{2}\cos(\omega)+\eta\sin(\omega)-\eta\cos
	(\omega),\label{examyuz}\\
	&  (\eta^{2}+\eta)\sin(\omega)+(\eta+1)\cos(\omega), \sqrt{3}(\omega+\eta^{2}-\eta))\nonumber
	\end{align}
	The surface (\ref{examyuz}) is shown in Figure 2 and \ the surface
	(\ref{examyuz}) with curve (\ref{xam1egri}) is shown in Figure 3.
\end{example}

%

\begin{figure}[h]
	\begin{center}
		\includegraphics[
		height=1.1in,
		width=3in
		]%
		{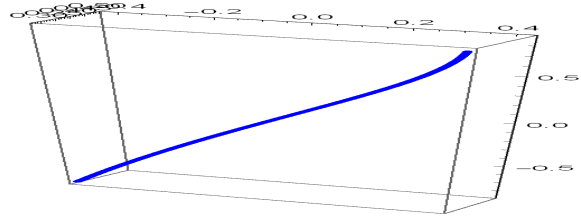}%
	\end{center}

\caption{Graph of $ \gamma $}
\end{figure}
\begin{figure}[h]
	\begin{center}
		\includegraphics[
		height=1.4in,
		width=3in
		]		{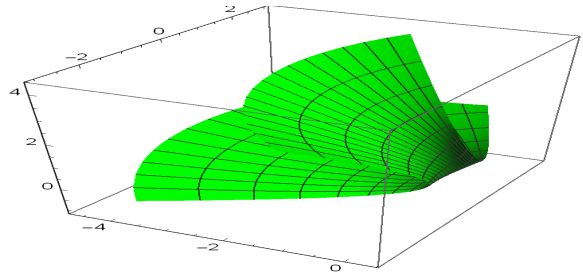}
	\end{center}
\caption{Graph of surface $ \Psi_1 $ is cunstructued by $ \gamma $}
\end{figure}\begin{figure}[h]
	\begin{center}
		\includegraphics[
		height=1.7in,
		width=3in
		]		{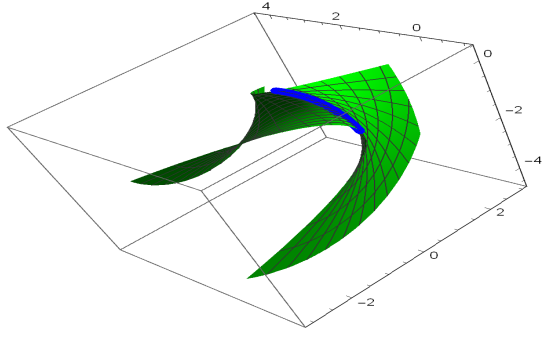}
	\end{center}
\caption{Graph of surface $ \Psi_1$ with curve $ \gamma $}
\end{figure}

\begin{example}
	Let the surface with an isoasymptotic curve above Example pass through the
	additional point $A_{2}=(\frac{1}{2},0,\frac{\sqrt{3}\pi}{4})$. For the
	convenience of calculations, taking $\omega_{2}=\frac{\pi}{2},$ $\eta_{2}=1,$ we have
	$A_{2}(\frac{1}{2},0,\frac{\sqrt{3}\pi}{4})$ and obtain the system of linear
	equations as follows:%
	\begin{align*}
	b_{1}+2b_{2} &  =-\frac{1}{4},\\
	a_{1}+2a_{2}+2\sqrt{3}(c_{2}+2c_{3}) &  =-7,\\
	\sqrt{3}(a_{1}+2a_{2}) &  =2(c_{2}+2c_{3}),\\
	b_{1} &  =-b_{2},\\
	-(a_{1}+a_{2}) &  =\sqrt{3}(c_{2}+c_{3}),\\
	\sqrt{3}(a_{1}+a_{2}) &  =c_{2}+c_{3}.
	\end{align*}
	So, we obtain%
	\begin{align*}
	b_{1}  & =\frac{1}{4},\text{ }b_{2}=-\frac{1}{4},\text{ }c_{2}=\frac{7\sqrt
		{3}(3-\sqrt{3})}{12},~c_{2}=\frac{7\sqrt{3}(\sqrt{3}-3)}{12},\\
	a_{1}  & =\frac{14(3-\sqrt{3})}{12},~a_{2}=\frac{14(\sqrt{3}-3)}{12}.
	\end{align*}
	Thus the surface with an isoasymptotic curve passing the two points
	$A_{1}=(1,7,\frac{\sqrt{3}\pi}{4})$ and $A_{2}(\frac{1}{2},0,\frac{\sqrt{3}%
		\pi}{4})$  is uniquely given by%

\begin{align}
\Psi_{2}(\omega,\eta) &  =(\frac{\sin(\omega)}{2}+\frac{7\left(  3-\sqrt
	{3}\right)  }{12}\left(  \eta-\eta^{2}\right)  \cos(\omega)\nonumber\\
&  -\frac{\sin(\omega)}{4}\left(  \eta-\eta^{2}\right)  +\frac{7\left(
	3-\sqrt{3}\right)  }{8}\left(  \eta^{2}-\eta^{3}\right)  \cos(\omega
),\nonumber\\
&  \frac{\cos(\omega)}{2}-\frac{7\left(  3-\sqrt{3}\right)  }{12}\left(
\eta-\eta^{2}\right)  \sin(\omega)\label{examyuz2}\\
&  -\left(  \eta-\eta^{2}\right)  \frac{\cos(\omega)}{4}-\frac{7\left(
	3-\sqrt{3}\right)  }{8}\left(  \eta^{2}-\eta^{3}\right)  \sin(\omega
)+,\nonumber\\
&  \frac{\sqrt{3}s}{2}+\frac{7\left(  \sqrt{3}-1\right)  }{4}\left(  \eta
-\eta^{2}\right)  -\frac{7\left(  \sqrt{3}-1\right)  }{8}\left(  \eta^{2}%
-\eta^{3}\right)  ).\nonumber
\end{align}
	The surface (\ref{examyuz2}) is shown in Figure 4 and \ the surface
	(\ref{examyuz2}) with curve (\ref{xam1egri}) is shown in Figure 5.
\end{example}
\begin{figure}
	[h]
	\begin{center}
		\includegraphics[
		height=1.6in,
		width=3in
		]%
		{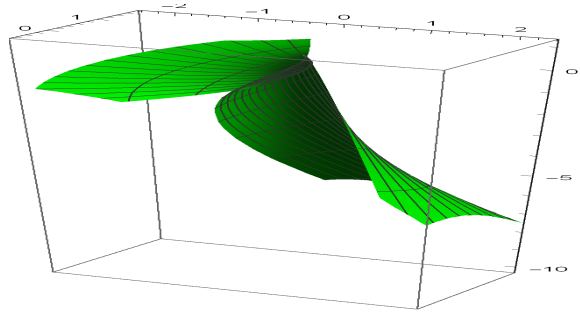}%
	\end{center}
\caption{Graph of surface $ \Psi_2$}
\end{figure}
\begin{figure}
	[h]
	\begin{center}
		\includegraphics[
		height=1.6in,
		width=3in
		]%
		{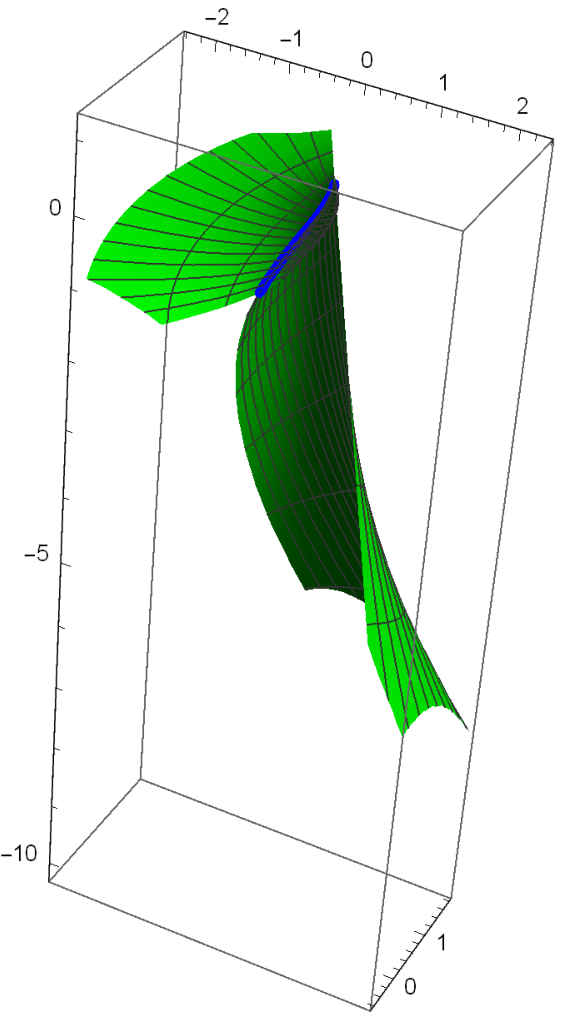}%
	\end{center}
\caption{Graph of surface $ \Psi_2$ with curve $ \gamma $}
\end{figure}

\end{document}